\documentclass{amsart}
\usepackage{amssymb,amscd,amsthm}
\usepackage{latexsym}
\date{\today}

\newcommand{\Z}{{\mathbb Z}}
\newcommand{\R}{{\mathbb R}}
\newcommand{\C}{{\mathbb C}}

\def\e{\varepsilon}

\newtheorem{theorem}{Theorem}

\newtheorem{lemma}{Lemma}
\newtheorem{prop}{Proposition}
\newtheorem{coro}{Corollary}

\begin{document}
\title[Quantum Dynamics via Complex Analysis Methods]{Quantum Dynamics via Complex Analysis Methods: General Upper Bounds Without Time-Averaging
and Tight Lower Bounds for the Strongly Coupled Fibonacci
Hamiltonian}

\author[D.\ Damanik]{David Damanik}

\address{Department of Mathematics, Rice University, Houston, TX~77005, USA}

\email{damanik@rice.edu}

\author[S.\ Tcheremchantsev]{Serguei Tcheremchantsev}

\address{UMR 6628--MAPMO, Universit\'{e} d'Orl\'{e}ans, B.P.~6759, F-45067 Orl\'{e}ans Cedex,
France}

\email{serguei.tcherem@labomath.univ-orleans.fr}

\thanks{D.\ D.\ was supported in part by NSF grant
DMS--0653720.}

\begin{abstract}
We develop further the approach to upper and lower bounds in
quantum dynamics via complex analysis methods which was introduced
by us in a sequence of earlier papers. Here we derive upper bounds
for non-time averaged outside probabilities and moments of the
position operator from lower bounds for transfer matrices at
complex energies. Moreover, for the time-averaged transport
exponents, we present improved lower bounds in the special case of
the Fibonacci Hamiltonian. These bounds lead to an optimal
description of the time-averaged spreading rate of the fast part
of the wavepacket in the large coupling limit. This provides the
first example which demonstrates that the time-averaged spreading
rates may exceed the upper box-counting dimension of the spectrum.
\end{abstract}

\maketitle

\section{Introduction}

This paper studies the long-time behavior of the solution to the
time-dependent Schr\"odinger equation, $i \partial_t \psi = H
\psi$, in the Hilbert space $\ell^2(\Z)$ with an initially
localized state, say $\psi(0) = \delta_0$. More precisely, the
Schr\"odinger operator is of the form
\begin{equation}\label{f.oper}
[H u](n) = u(n+1) + u(n-1) + V(n) u(n),
\end{equation}
where $V : \Z \to \R$ is the potential, and the solution is given
by
\begin{equation}\label{schdyn}
\psi(t) = e^{-itH} \delta_0.
\end{equation}
The probability of finding the state in $\{ n \in \Z : n \ge N \}$
at time $t$ is given by
$$
P_r(N,t) = \sum_{n \ge N} | \langle e^{-itH} \delta_0 , \delta_n
\rangle |^2.
$$
Similarly,
$$
P_l(N,t) = \sum_{n \le -N} | \langle e^{-itH} \delta_0 , \delta_n
\rangle |^2
$$
is equal to the probability of finding the state in $\{ n \in \Z :
n \le -N \}$ at time $t$.

It is in general a hard problem to bound these so-called outside
probabilities from above. Our recent paper \cite{dt3} introduced a
way of estimating their time-averages from above in terms of the
norms of transfer matrices at complex energies. In this paper we
show how similar estimates can be obtained without the need to
take time-averages.

To state these results, let us recall the definition of the
transfer matrices. For $n \in \Z$ and $z \in \C$, define the
transfer matrix $\Phi(n,z)$ by
\begin{equation}
\Phi(n,z) = \begin{cases} T(n,z) \cdots T(1,z) & n \ge 1, \\
\mathrm{Id} & n = 0,
\\ [T(n+1,z)]^{-1} \cdots [T(0,z)]^{-1} & n \le -1, \end{cases}
\label{Phi}
\end{equation}
where
$$
T(m,z) = \left( \begin{array}{cr} z - V(m) & -1 \\ 1 & 0
\end{array} \right).
$$
The definition is such that $u : \Z \to \C$ solves
\begin{equation}\label{f.eve}
u(n+1) + u(n-1) + V(n)u(n) = z u(n)
\end{equation}
if and only if
$$
\left( \begin{array}{c} u(n+1) \\ u(n) \end{array} \right) =
\Phi(n,z) \left(
\begin{array}{c} u(1) \\ u(0) \end{array} \right)
$$
for every $n \in \Z$.

\begin{theorem}\label{t.main}
Suppose $H$ is given by \eqref{f.oper}, where $V$ is a bounded
real-valued function, and $K \ge 4$ is such that $\sigma (H)
\subseteq [-K+1,K-1]$. Then, the outside probabilities can be
bounded from above in terms of transfer matrix norms as follows:
\begin{align*}
P_r(N,t) & \lesssim \exp (-c N) + t^4 \int_{-K}^K \left( \max_{0
\le n \le
N-1} \left\| \Phi \left( n,E + i t^{-1} \right) \right\|^2 \right)^{-1} dE, \\
P_l(N,t) & \lesssim \exp (-c N) + t^4 \int_{-K}^K \left(
\max_{-N+1 \le n \le 0} \left\| \Phi \left( n,E + i t^{-1} \right)
\right\|^2 \right)^{-1} dE.
\end{align*}
The implicit constants depend only on $K$.
\end{theorem}

This result, which is the non-time averaged analogue of
\cite[Theorem~7]{dt3}, has a number of consequences akin to those
discussed in \cite{dt3}. For $p > 0$, consider the $p$-th moment
of the position operator,
$$
\langle |X|_{\delta_0}^p \rangle (t) = \sum_{n \in \Z} |n|^p |
\langle e^{-itH} \delta_0 , \delta_n \rangle |^2
$$
and the upper and lower transport exponents
$\beta^+_{\delta_0}(p)$ and $\beta^-_{\delta_0}(p)$, given,
respectively, by
$$
\beta^+_{\delta_0}(p)=\limsup_{t \to \infty} \frac{\log \langle
|X|_{\delta_0}^p \rangle (t) }{p \, \log t}.
$$
and
$$
\beta^-_{\delta_0}(p)=\liminf_{t \to \infty} \frac{\log \langle
|X|_{\delta_0}^p \rangle (t) }{p \, \log t}.
$$
Let us briefly discuss a connection with the outside
probabilities. Set
$$
P(N,t) = P_l(N,t) + P_r(N,t).
$$
Following \cite{GKT}, for $0 \le \alpha \le \infty$, define
\begin{equation}\label{smindef}
S^-(\alpha) = - \liminf_{t \to \infty} \frac{\log P(t^\alpha - 1,
t) }{\log t}
\end{equation}
and
\begin{equation}\label{spludef}
S^+(\alpha) = - \limsup_{t \to \infty} \frac{\log P(t^\alpha - 1,
t) }{\log t}.
\end{equation}
For every $\alpha$, $0 \le S^+ (\alpha) \le S^- (\alpha) \le
\infty$.

These numbers control the power decaying tails of the wavepacket.
In particular, the following critical exponents are of interest:
\begin{align}
\alpha_l^\pm & = \sup \{ \alpha \ge 0  :  S^\pm (\alpha)=0 \}, \label{aldef} \\
\alpha_u^\pm & = \sup \{ \alpha \ge 0  :  S^\pm (\alpha) < \infty
\}. \label{audef}
\end{align}
We have that $0 \le \alpha_l^- \le \alpha_u^- \le 1$, $0 \le
\alpha_l^+ \le \alpha_u^+ \le 1$, and also that $\alpha_l^- \le
\alpha_l^+$, $\alpha_u^- \le \alpha_u^+$. One can interpret
$\alpha_l^\pm$ as the (lower and upper) rates of propagation of
the essential part of the wavepacket, and $\alpha_u^\pm$ as the
rates of propagation of the fastest (polynomially small) part of
the wavepacket; compare \cite{GKT}. In particular, if
$\alpha>\alpha_u^+$, then $P(t^\alpha, t)$ goes to $0$ faster than
any inverse power of $t$. Since a ballistic upper bound holds in
our case (for any potential $V$), a slight modification of
\cite[Theorem 4.1]{GKT} yields
$$
\lim_{p \to 0} \beta_{\delta_0}^\pm (p) = \alpha_l^\pm
$$
and
$$
\lim_{p \to \infty} \beta_{\delta_0}^\pm (p) = \alpha_u^\pm.
$$
In particular, since $\beta^\pm_{\delta_0} (p)$ are nondecreasing,
we have that
\begin{equation}\label{betaalphabound}
\beta_{\delta_0}^\pm (p) \le \alpha_u^\pm \quad \text{ for every }
p > 0.
\end{equation}

\begin{coro}\label{c.main}
Suppose $H$ is given by \eqref{f.oper}, where $V$ is a bounded
real-valued function, and $K \ge 4$ is such that $\sigma (H)
\subseteq [-K+1,K-1]$. Suppose that, for some $C \in (0,\infty)$
and $\alpha \in (0,1)$, we have
\begin{equation}\label{assumeright}
\int_{-K}^K \left( \max_{1 \le n \le C t^\alpha} \left\| \Phi
\left( n, E + i t^{-1} \right) \right\|^2 \right)^{-1} dE =
O(t^{-m})
\end{equation}
and
\begin{equation}\label{assumeleft}
\int_{-K}^K \left( \max_{1 \le -n \le C t^\alpha} \left\| \Phi
\left( n, E + i t^{-1} \right) \right\|^2 \right)^{-1} dE =
O(t^{-m})
\end{equation}
for every $m \ge 1$. Then
\begin{equation}\label{apubound}
\alpha_u^+ \le \alpha.
\end{equation}
In particular,
\begin{equation}\label{bppbound}
\beta^+_{\delta_0} (p) \le \alpha \quad \text{ for every } p > 0.
\end{equation}
\end{coro}

Theorem~\ref{t.main} and Corollary~\ref{c.main} provide the
analogues of the central general results from \cite{dt3} for
quantities that are not time-averaged. For comparison purposes,
let us introduce the following notation. If $f(t)$ is a function
of $t > 0$ and $T > 0$ is given, we denote the time-averaged
function at $T$ by $\langle f \rangle (T)$:
$$
\langle f \rangle (T) = \frac{2}{T} \int_0^{\infty} e^{-2t/T} f(t)
\, dt.
$$
Thus, we consider, for example, $\langle P_l(N,\cdot) \rangle
(T)$, $\langle P_r(N,\cdot) \rangle (T)$, and $\langle \langle
|X|_{\delta_0}^p \rangle \rangle (T)$. For these time-averaged
quantities, we can then define the transport exponents $\langle
\beta^\pm_{\delta_0}(p) \rangle$ and their limiting values
$\langle \alpha_l^\pm \rangle$ and $\langle \alpha_l^\pm \rangle$
in the same way as above. For example, if in the formulation of
Corollary~\ref{c.main}, we replace $\alpha_u^+$ by $\langle
\alpha_u^+ \rangle$ and $\beta^+_{\delta_0} (p)$ by $\langle
\beta^+_{\delta_0} (p) \rangle$, we obtain the assertion of
\cite[Theorem~1]{dt3}.

\bigskip

Let us now turn to a discussion of a special case. The Fibonacci
Hamiltonian is the discrete one-dimensional Schr\"odinger operator
in $\ell^2(\Z)$ as in \eqref{f.oper} with potential $V : \Z \to
\R$ is given by
\begin{equation}\label{fibpot}
V(n) = \lambda \chi_{[1-\phi^{-1},1)}(n \phi^{-1} \! +
\theta\!\!\! \mod 1).
\end{equation}
Here, $\lambda > 0$ is the coupling constant, $\phi$ is the golden
mean,
$$
\phi = \frac{\sqrt{5}+1}{2}
$$
and $\theta \in [0,1)$ is the phase. This is the most prominent
model of a one-dimensional quasicrystal; compare the survey
articles \cite{d,su3}. It is known that the spectrum of $H$ is
independent of $\theta$ \cite{bist}; let us denote it by
$\Sigma_\lambda$. Moreover, the Lebesgue measure of
$\Sigma_\lambda$ is zero \cite{su2} and all spectral measures are
purely singular continuous \cite{dl}.

The quantum evolution with $H$ given by the Fibonacci Hamiltonian
has been studied in many papers. It had long been expected to be
anomalous in the sense that it is markedly different from the
behavior in the periodic case (leading to ballistic transport) and
the random case (leading to dynamical localization); see, for
example, papers in the physics literature by Abe and Hiramoto
\cite{ah,ha}. Lower bounds, showing in particular the absence of
dynamical localization, were shown in
\cite{d1,dkl,dst,dt1,dt2,jl,kkl}. There are far fewer paper
establishing upper bounds for this model, especially for
quantities like the moments of the position operator. Killip et
al.\ showed for $\theta = 0$ and $\lambda \ge 8$ that the slow
part of the wavepacket does not move ballistically \cite{kkl}.
Their result was extended to general $\theta$ in \cite{d2}. The
first result establishing for $\lambda \ge 8$ bounds from above
for the whole wavepacket, and hence quantities like the moments of
the position operator, is contained in \cite{dt3}. In that paper
only the case $\theta = 0$ is studied but it is remarked that the
ideas from \cite{d2} will allow one to treat general $\theta$'s.

The paper \cite{dt3} introduced a new tool, the complex trace map,
that allows one to use complex analysis methods in a context where
real analysis methods were used earlier. This was important in our
proof of non-trivial upper bounds for the (time-averaged)
transport exponents. Given the analysis of \cite{dt3} and
Corollary~\ref{c.main} above, which strengthens
\cite[Theorem~1]{dt3}, we obtain the following strengthening of
\cite[Theorem~3]{dt3}.

\begin{theorem}\label{t.main2}
Consider the Fibonacci Hamiltonian, that is, the operator
\eqref{f.oper} with potential \eqref{fibpot}. Assume that $\lambda
\ge 8$ and let
$$
\alpha(\lambda) = \frac{2 \log \phi}{\log S_l(\lambda)}.
$$
with
\begin{equation}\label{f.xilambda}
S_l(\lambda) = \frac{1}{2} \left( (\lambda - 4) + \sqrt{(\lambda -
4)^2 - 12} \right).
\end{equation}
Then, $\alpha_u^+ \le \alpha (\lambda)$, and hence $\beta^+ (p)
\le \alpha (\lambda)$ for every $p > 0$.
\end{theorem}

As a byproduct of our study of the complex trace map in
\cite{dt3}, we established a distortion result that is useful to
bound the transport exponents from either side. Since \cite{dt3}
focused on quantum dynamical upper bounds, we present the
application of the distortion result to quantum dynamical lower
bounds here. For this result, we still need to consider
time-averaged quantities!

\begin{theorem}\label{t.main3}
Consider the Fibonacci Hamiltonian, that is, the operator
\eqref{f.oper} with potential \eqref{fibpot}. Suppose $\lambda >
\sqrt{24}$ and let
$$
S_u(\lambda) = 2\lambda + 22.
$$
We have
\begin{equation}\label{betalower}
\langle \beta^-_{\delta_0} (p) \rangle \ge \frac{2 \log \phi}{\log
S_u(\lambda)} - \frac{2}{p}\left( 1 + \frac{C \log \lambda}{\log
S_u(\lambda)} \right)
\end{equation}
for a suitable constant $C$, and therefore
\begin{equation}\label{alphalower}
\langle \alpha_u^- \rangle \ge \frac{2 \log \phi}{\log
S_u(\lambda)}.
\end{equation}
\end{theorem}

\noindent\textit{Remark.} In the special case $\theta = 0$, the
lower bound for $\langle \beta^\pm_{\delta_0} (p) \rangle$ can be
improved; see Theorem~\ref{t.main4} at the end of this paper.

\bigskip

The best previously known lower bound for $\langle \alpha_u^-
\rangle$ in the large coupling regime was obtained in \cite{DEGT}.
It reads
\begin{equation}\label{alphadim}
\langle \alpha_u^- \rangle \ge \dim_B^\pm (\Sigma_\lambda).
\end{equation}
Here, $\dim_B^\pm (\Sigma_\lambda)$ denotes the upper/lower box
counting dimension of $\Sigma_\lambda$. Such a bound holds
whenever the transfer matrices are polynomially bounded on the
spectrum, which in particular holds in the Fibonacci case.

The authors of \cite{DEGT} performed a detailed study of these
dimensions. A particular consequence of their study is the
following asymptotic statement,
\begin{equation}\label{dimasymp}
\lim_{\lambda \to \infty} \dim_B^\pm ( \Sigma_\lambda ) \cdot \log
\lambda = f^\# \log \phi,
\end{equation}
where $f^\#$ is an explicit constant (the unique maximum of an
explicit function) that is roughly given by
$$
f^\# \approx 1.83156.
$$

On the other hand, for $\lambda \ge 8$, \cite[Theorem~3]{dt3} (see
also Theorem~\ref{t.main2} above) gives the following upper bound
for $\langle \alpha_u^+ \rangle$,
\begin{equation}\label{alphaupper}
\langle \alpha_u^+ \rangle \le  \frac{2 \log \phi}{\log
S_l(\lambda)}.
\end{equation}
Thus, combining \eqref{alphalower} and \eqref{alphaupper}, we find
the following exact asymptotic result.

\begin{coro}
For the Fibonacci Hamiltonian, we have
\begin{equation}\label{alphaasymp}
\lim_{\lambda \to \infty} \langle \alpha_u^\pm \rangle \cdot \log
\lambda = 2 \log \phi.
\end{equation}
\end{coro}

A particular consequence is that, as $\lambda \to \infty$, the
limit behavior is the same for both $\langle \alpha_u^+ \rangle$
and $\langle \alpha_u^- \rangle$. It would be of interest to show
that these quantities are in fact equal for finite $\lambda$.

Comparing the asymptotic results \eqref{dimasymp} and
\eqref{alphaasymp}, we see that the Fibonacci Hamiltonian at large
coupling may serve as an example where the inequality in
\eqref{alphadim} is strict. This result is also relevant to a
question raised by Last in \cite[Section~9]{l} who asked whether
the upper box counting dimension of the spectrum could serve as an
upper bound for dynamics and suggested that the expected answer is
negative in general. See \cite{GM,WA} for numerical results
providing evidence supporting this expectation.

Let us put the recent results for the Fibonacci Hamiltonian in
perspective. The present paper and \cite{DEGT} result from an
attempt to describe dimensions and transport exponents
\textit{exactly}. This is certainly a challenging problem and we
have precise results only in an asymptotic regime; compare
\eqref{dimasymp} and \eqref{alphaasymp}. It would certainly be of
interest to prove exact results also for fixed finite $\lambda$.
Moreover, the behavior in the small coupling regime is poorly
understood. Many results that hold for $\lambda$ above some
critical coupling do not obviously extend to smaller values since
parts of the proofs break down. There is a definite need for new
insights in order to prove results at small coupling.

\section{Upper Bounds for Outside Probabilities Without
Time-Averaging}\label{s.2}

Our goal in this section is to prove Theorem~\ref{t.main} and
Corollary~\ref{c.main}. In particular, we will see how a specific
consequence of the Dunford functional calculus allows us to
replace the Parseval formula, which was a crucial ingredient in
several earlier papers on quantum dynamical bounds
\cite{dst,dt1,dt2,dt3,GKT,kkl}. This observation is the key to
obtaining upper bounds for non-averaged quantities.

\begin{lemma}\label{l.1}
We have
$$
\langle e^{-itH} \delta_0 , \delta_n \rangle = - \frac{1}{2 \pi i}
\int_\Gamma e^{-itz} \langle (H - z)^{-1} \delta_0 , \delta_n
\rangle \, dz
$$
for every $n \in \Z$, $t \in \R$, and positively oriented simple
closed contour $\Gamma$ in $\C$ that is such that the spectrum of
$H$ lies inside $\Gamma$.
\end{lemma}

\begin{proof}
This is a consequence of the so-called Dunford functional
calculus; see \cite{dun} and also \cite{dunsch,reesim}.
\end{proof}

\begin{lemma}\label{l.2}
Suppose $H$ is given by \eqref{f.oper}, where $V$ is a bounded
real-valued function, and $K \ge 4$ is such that $\sigma (H)
\subseteq [-K+1,K-1]$. Then,
\begin{align*}
P_r(N,t) & \lesssim \exp (-c N) + \int_{-K}^K \sum_{n \ge N}  |
\langle (H - E - i t^{-1})^{-1} \delta_0 , \delta_n
\rangle |^2 dE, \\
P_l(N,t) & \lesssim \exp (-c N) + \int_{-K}^K \sum_{n \le -N}  |
\langle (H - E - i t^{-1})^{-1} \delta_0 , \delta_n \rangle |^2
dE.
\end{align*}
\end{lemma}

\begin{proof}
Given $t > 0$, we will consider the following contour $\Gamma$:
$\Gamma=\Gamma_1 \cup \Gamma_2 \cup \Gamma_3 \cup \Gamma_4$, where
\begin{align*}
\Gamma_1 & = \left\{ z = E + iy : E \in [-K,K], \ y = t^{-1} \right\} \\
\Gamma_2 & = \left\{ z = E + iy : E = -K, \ y \in [-1,t^{-1}] \right\} \\
\Gamma_3 & = \left\{ z = E + iy : E \in [-K,K], \ y = -1 \right\} \\
\Gamma_4 & = \left\{ z = E + iy : E = K, \ y \in [-1,t^{-1}]
\right\}.
\end{align*}
Notice that for $z \in \Gamma$, we have $\Im z \le t^{-1}$ and
hence $|e^{-itz}| = e^{t \Im z} \le e$. Thus, by Lemma~\ref{l.1},
$$
| \langle e^{-itH} \delta_0 , \delta_n \rangle | \lesssim \sum_{j
= 1}^4 \int_{\Gamma_j} \left| \langle (H - z)^{-1} \delta_0 ,
\delta_n \rangle \right| \, |dz|.
$$

If $z \in \Gamma_2 \cup \Gamma_3 \cup \Gamma_4$, then clearly
$\mathrm{dist} (z, \sigma (H)) \ge 1$ and hence the well-known
Combes-Thomas estimate allows us to bound the contributions from
$\Gamma_2, \Gamma_3, \Gamma_4$ to $P_r(N,t)$ and $P_l(N,t)$ by $C
\exp(-cN)$.

The integral over $\Gamma_1$ can be estimated using the
Cauchy-Schwarz inequality:
$$
\left( \int_{\Gamma_1} \left| \langle (H - z)^{-1} \delta_0 ,
\delta_n \rangle \right| \, |dz| \right)^2 \le C(K) \int_{-K}^K
\left| \langle (H - E - it^{-1})^{-1} \delta_0 , \delta_n \rangle
\right|^2 dE.
$$
Combining these estimates, we obtain the assertion of the lemma.
\end{proof}

\begin{proof}[Proof of Theorem~\ref{t.main}.]
We have the following estimates \cite[p.~811]{dt3}:
\begin{align}
\label{f.dt3form1} \langle P_r \!(N,\cdot) \rangle(T) & \lesssim
\exp (-c N) + T^{-1} \int_{-K}^K \sum_{n \ge N} | \langle (H - E -
i T^{-1})^{-1} \delta_0 , \delta_n
\rangle |^2 dE, \\
\label{f.dt3form2} \langle P_l(N,\cdot) \rangle (T) & \lesssim
\exp (-c N) + T^{-1} \int_{-K}^K \sum_{n \le -N} | \langle (H - E
- i T^{-1})^{-1} \delta_0 , \delta_n \rangle |^2 dE.
\end{align}
It was then shown in \cite{dt3} how to derive
\begin{align*}
\langle P_r(N,\cdot) \rangle(T) & \lesssim \exp (-c N) + T^3
\int_{-K}^K \left( \max_{0 \le n \le
N-1} \left\| \Phi \left( n,E + i T^{-1} \right) \right\|^2 \right)^{-1} dE, \\
\langle P_l(N,\cdot) \rangle (T) & \lesssim \exp (-c N) + T^3
\int_{-K}^K \left( \max_{-N+1 \le n \le 0} \left\| \Phi \left( n,E
+ i T^{-1} \right) \right\|^2 \right)^{-1} dE.
\end{align*}
from \eqref{f.dt3form1} and \eqref{f.dt3form2}. If we use
Lemma~\ref{l.2} instead of \eqref{f.dt3form1} and
\eqref{f.dt3form2} and then follow the very same steps, we obtain
the desired bounds
\begin{align*}
P_r(N,t) & \lesssim \exp (-c N) + t^4 \int_{-K}^K \left( \max_{0
\le n \le
N-1} \left\| \Phi \left( n,E + i t^{-1} \right) \right\|^2 \right)^{-1} dE, \\
P_l(N,t) & \lesssim \exp (-c N) + t^4 \int_{-K}^K \left(
\max_{-N+1 \le n \le 0} \left\| \Phi \left( n,E + i t^{-1} \right)
\right\|^2 \right)^{-1} dE.
\end{align*}
This concludes the proof.
\end{proof}

\begin{proof}[Proof of Corollary~\ref{c.main}.]
Let us choose $N(t) = \lfloor C t^\alpha \rfloor $, where $C \in
(0,\infty)$ and $\alpha \in (0,1)$ are chosen such that
\eqref{assumeright} and \eqref{assumeleft} hold. Observe that
$P(N(t), t)=P(\lceil N(t) \rceil , t)$. Then Theorem~\ref{t.main}
shows that $P_r(N(t),t)$ and $P_l(N(t),t)$ go to $0$ faster than
any inverse power of $t$. By definition of $S^+(\alpha)$ and
$\alpha_u^+$ (cf.~\eqref{spludef} and \eqref{audef}), it follows
that $\alpha_u^+ \le \alpha$, which is \eqref{apubound}. Finally,
\eqref{bppbound} follows from \eqref{betaalphabound}.
\end{proof}

\section{Tight Lower Bounds for Strongly Coupled
Fibonacci}\label{s.3}

\subsection{The Complex Trace Map and the Distortion of Balls}

In this subsection we recall some ideas from \cite{dt3} and
present an improvement of the key distortion result contained in
that paper.

For $z \in \C$ and $n \in \Z$, consider the transfer matrices
$\Phi(n,z)$ associated with the difference equation \eqref{f.eve}
where $V$ is the Fibonacci potential given by \eqref{fibpot}.
Notice that these matrices depend on both $\lambda$ and $\theta$.

Define the matrices $M_k(z)$ by
\begin{equation}\label{fibtransfer}
\Phi_{\theta = 0}(F_k,z) = M_k(z), \quad k \ge 1,
\end{equation}
where $F_k$ is the $k$-th Fibonacci number, that is, $F_0 = F_1 =
1$ and $F_{k+1} = F_k + F_{k-1}$ for $k \ge 1$.

It is well-known that
$$
M_{k+1}(z) = M_{k-1}(z) M_k(z), \quad k \ge 2.
$$
For the variables $x_k(z) = \frac{1}{2} \mathrm{Tr} \, M_k(z))$,
$k \ge 1$, we have the recursion
\begin{equation}\label{tmrecursion}
x_{k+1}(z) = 2 x_{k}(z) x_{k-1}(z) - x_{k-2}(z)
\end{equation}
and the invariant
\begin{equation}\label{inv}
x_{k+1}(z)^2 + x_k(z)^2 + x_{k-1}(z)^2 - 2x_{k+1}(z) x_k(z)
x_{k-1}(z) - 1 \equiv \frac{\lambda^2}{4}.
\end{equation}
Letting $x_{-1}(z) = 1$ and $x_0(z) = z/2$, the recursion
\eqref{tmrecursion} holds for all $k \ge 0$.

For $\delta \ge 0$, consider the sets
$$
\sigma_k^\delta = \{ z \in \C : |x_k(z)| \le 1 + \delta \}.
$$
We have
$$
\sigma_k^\delta \cup \sigma_{k-1}^\delta \supseteq
\sigma_{k+1}^\delta \cup \sigma_k^\delta \to \Sigma_\lambda.
$$
Assume $\lambda > \lambda_0(\delta)$, where
\begin{equation}\label{lambda0}
\lambda_0(\delta) = [12(1 + \delta)^2 + 8(1 + \delta)^3 + 4]^{1/2}.
\end{equation}
The invariant \eqref{inv} implies
\begin{equation}\label{disj}
\sigma_k^\delta \cap \sigma_{k+1}^\delta \cap \sigma_{k+2}^\delta =
\emptyset.
\end{equation}
Moreover, the set $\sigma_k^\delta$ has exactly $F_k$ connected
components. Each of them is a topological disk that is symmetric
about the real axis.

It is known that all roots of $x_k$ are real. Consider such a root
$z$, that is, $x_k(z) = 0$ and define
$$
m (z) = \# \{ 0 \le l \le k-1 : |x_l(z)| \le 1 \}.
$$
Let
$$
c_{k,m} = \# \{ \text{roots of $x_k$ with $m(z) = m$} \}.
$$
An explicit formula for $c_{k,m}$ was found in \cite[Lemma~5]{DEGT}
(noting that our $c_{k,m}$ equals $a_{k,m} + b_{k,m}$ in the
notation of that paper). All we need here is the following
consequence of this result.

\begin{lemma}\label{supplemm}
For every $k \ge 2$, $c_{k,m}$ is non-zero if and only if
$\tfrac{k}{2} \le m \le \tfrac{2k}{3}$.
\end{lemma}

Let $\{z_k^{(j)}\}_{1 \le j \le F_k}$ be the zeros of $x_k$ and
write $m_k^{(j)} = m(z_k^{(j)})$ for $1 \le j \le F_k$. Denote
$B(z,r) = \{ w \in \C : | w - z | < r \}$. Then, we have the
following distortion result.

\begin{prop}\label{distortion}
Fix $k \ge 3$, $\delta > 0$, and $\lambda > \max \{
\lambda_0(2\delta) , 8 \}$. Then, there are constants
$c_\delta,d_\delta
> 0$ such that
\begin{equation}\label{unionballs}
\bigcup_{j = 1}^{F_k} B \left( z_k^{(j)}, r_k^{(j)} \right)
\subseteq \sigma_k^\delta \subseteq \bigcup_{j = 1}^{F_k} B \left(
z_k^{(j)}, R_k^{(j)} \right),
\end{equation}
where $r_k^{(j)} = c_\delta S_u(\lambda)^{-m_k^{(j)}}$, and
$R_k^{(j)} = d_\delta S_l(\lambda)^{-m_k^{(j)}}$. The first
inclusion in \eqref{unionballs} only needs the assumption $\lambda
> \lambda_0(2\delta)$.
\end{prop}

\begin{proof}
The proof of this result is analogous to the proof of
\cite[Proposition~3]{dt3}. As explained there, $m_k^{(j)}$ governs
the size of $|x_k'(z_k^{(j)})|$, which in turn is closely related
to the size and shape of the connected component of
$\sigma_k^\delta$ that contains $z_k^{(j)}$ by Koebe's Distortion
Theorem.
\end{proof}

\subsection{Power-Law Upper Bounds for Transfer Matrices}

In this subsection we prove power-law upper bounds for the norm of
the transfer matrices for suitable complex energies and suitable
maximal scales.

\begin{prop}\label{powertm}
For every $\lambda , \delta > 0$, there are constant $C,\gamma$
such that for every $k$, every $z \in \sigma_k^\delta$, and every
$N$ with $1 \le N \le F_k$, we have
\begin{equation}\label{powerlaw}
\| \Phi(N,z) \| \le C N^\gamma.
\end{equation}
The constant $\gamma$ behaves like $O(\log \lambda)$ as $\lambda
\to \infty$ for each fixed $\delta$.
\end{prop}

\begin{proof}
For $\theta = 0$, a modification of the proof of
\cite[Proposition~3.2]{dt1} gives the result. The extension to
arbitrary $\theta$ then follows using ideas from \cite{dl2}. For
the convenience of the reader, we sketch the main parts of the
argument.

Using the invariant \eqref{inv} and \cite[Lemma~4]{dt3}, it is
readily seen that $z \in \sigma_k^\delta$ implies the uniform
bound
\begin{equation}\label{tracebound}
|x_j(z)| \le C_{\lambda,\delta} \qquad \text{ for } 0 \le j \le
k-1,
\end{equation}
where $C_{\lambda,\delta}$ is an explicit constant that behaves
like $O(\lambda)$ as $\lambda \to \infty$; compare
\cite[Lemma~3.1]{dt1}.

Next, the uniform trace bound \eqref{tracebound} yields an upper
bound for the matrix norm,
\begin{equation}\label{normbound}
\|M_j(z)\| \le \left( C_{\lambda,\delta} \right)^j \qquad \text{
for } 0 \le j \le k-1,
\end{equation}
compare \cite[Eq.~(48)]{dt1}.

The final step is an interpolation of \eqref{normbound}. Given any
transfer matrix, $\Phi(N,z)$, it is possible to write it as a
product of matrices of type $M_j(z)$; compare
\cite[Section~3]{dl}. A careful analysis of the factors that
occur, together with the estimate \eqref{normbound}, then gives
\eqref{powerlaw}. The bound on $\gamma$ follows as in
\cite[Proposition~3.2]{dt1}.
\end{proof}

\subsection{Lower Bounds for the Transport Exponents}

In this subsection we prove Theorem~\ref{t.main3} using the
results from the previous two sections.

\begin{proof}[Proof of Theorem~\ref{t.main3}.]
By assumption, $\lambda > \sqrt{24}$. Thus, it is possible to
choose $\delta > 0$ so that $\lambda > \lambda_0(2 \delta)$.
Recall that $\sigma_k^{\delta}=\{ z \in \C \ : \ |x_k(z)| \le
1+\delta \}$. It follows from Lemma~\ref{supplemm} and
Proposition~\ref{distortion} that $\sigma_k^{\delta}$ has a
connected component $D_k$ such that
$$
B(z_k, r_k) \subset D_k,
$$
where $z_k \in \R$ and
$$
r_k = c_{\delta} S_u(\lambda)^{-\lceil \frac{k}{2} \rceil}.
$$

It follows from Proposition~\ref{powertm} that for $z \in D_k$ and
$1 \le N \le F_k$, we have
$$
\|\Phi(N,z)\| \le  C N^{\gamma}
$$
with suitable constants $C$ and $\gamma$.

Let us fix $k$ and take $N = F_k \sim \phi^k$ and $\e =
\frac{1}{T} \le \frac{r_k}{2}$. Thus,
$$
T \ge \frac{2}{r_k} = \frac{2}{c_\delta} S_u(\lambda)^{\lceil
\frac{k}{2} \rceil} \gtrsim N^s,
$$
where
$$
s = \frac{\log S_u(\lambda)}{2 \log \phi}.
$$

Due to the Parseval formula (see, e.g., \cite[Lemma~3.2]{kkl}), we
can bound the time-averaged outside probabilities from below as
follows,
\begin{equation}\label{parseval}
\langle P(N,\cdot) \rangle (T) \gtrsim \e \int_\R
\|\Phi(N,E+i\e)\|^{-2} \, dE.
\end{equation}
To bound the integral from below, we integrate only over those $E
\in (z_k-r_k, z_k+r_k)$ for which $E+i\e \in B(z_k, r_k) \subset
D_k$. The length of such an interval $I_k$ is larger than $cr_k$
for some suitable $c>0$. For $E \in I_k$, we have
$$
\|\Phi(N, E+i\e)\| \lesssim N^{\gamma} \lesssim
T^{\frac{\gamma}{s}}.
$$
Therefore, \eqref{parseval} gives
\begin{equation}\label{eq1}
\langle P(N,\cdot) \rangle (T) \gtrsim \frac{r_k}{T} \,
T^{-\frac{2\gamma}{s}} \gtrsim T^{-2-\frac{2\gamma}{s}},
\end{equation}
where $N = F_k$, $T \ge N^s$, for any $k$.

Now let us take any $T \ge 1$ and choose $k$ maximal with $F_k^s
\le T$. Then,
$$
F_k^s \le T < F_{k+1}^s \le 2^s F_k^s.
$$
It follows from \eqref{eq1} that
$$
\langle P(\tfrac12 T^\frac{1}{s},\cdot) \rangle (T) \ge \langle
P(N,\cdot) \rangle (T) \gtrsim T^{-2-\frac{2\gamma}{s}}
$$
for all $T \ge 1$. It follows from the definition of $\langle
\beta^-(p) \rangle$ and $\langle \alpha_u^- \rangle$ that
$$
\langle \beta_{\delta_0}^-(p) \rangle \ge \frac{1}{s} -
\frac{2}{p}\left( 1 + \frac{\gamma}{s} \right)
$$
and
$$
\langle \alpha_u^- \rangle \ge \frac{1}{s}.
$$
Since
$$
\frac{1}{s} = \frac{2 \log \phi}{\log S_u(\lambda)},
$$
and $\gamma = O(\log \lambda)$, this completes the proof of
Theorem~\ref{t.main3}.
\end{proof}

\subsection{Improved Lower Bounds for Zero Phase}

We conclude with a discussion of the special case $\theta = 0$.
For this particular value of the phase, the result above can be
improved. The key change is that in this case, one has ``all the
possible squares adjacent to the origin.'' This then enables one
to directly estimate $\langle P(N,\cdot) \rangle (T)$ from below
without employing the general upper bounds for the transfer
matrices.

\begin{theorem}\label{t.main4}
Consider the operator \eqref{f.oper} with potential \eqref{fibpot}
and phase $\theta = 0$. Suppose $\lambda > \sqrt{24}$ and let
$$
S_u(\lambda) = 2\lambda + 22.
$$
We have
\begin{equation}\label{betalower0}
\langle \beta^-_{\delta_0} (p) \rangle \ge \frac{2 \log \phi}{\log
S_u(\lambda)} - \frac{2}{p}\left( 1 - \frac{2 \log
\frac{\sqrt{17}}{4}}{5\log S_u(\lambda)} \right)
\end{equation}
for every $p > 0$.
\end{theorem}

\begin{proof}
Starting again with the Parseval formula \cite[Lemma~3.2]{kkl}, we
have
\begin{equation}\label{parseval0}
\langle P_r(N,\cdot) \rangle (T) = \frac{1}{\pi T} \sum_{n \ge N}
\int_\R | \langle (H - E - i T^{-1})^{-1} \delta_0 , \delta_n
\rangle |^2 \, dE.
\end{equation}
Using the fact that
$$
u(n) = \langle (H - E - i T^{-1} \delta_0 , \delta_n \rangle
$$
solves the difference equation $H u = (E + i T^{-1}) u$ away from
zero, one can bound the right-hand side of \eqref{parseval0} from
below by the Gordon-type mass reproduction technique used, for
example, in \cite{d1,dkl,dt2,jl}. We refer the reader to
\cite[Lemmas~2 and 3]{dt2} for a formulation of this technique
that is particularly well suited to our situation here. It is
important to point out that while in those papers, only real
energies are considered, the statements extend to complex energies
as soon as one can arrange for the same transfer matrix trace
estimates to hold.

As in the proof of Theorem~\ref{t.main3}, since $\lambda >
\sqrt{24}$, it is possible to choose $\delta > 0$ so that $\lambda
> \lambda_0(2 \delta)$. Fixing this value of $\delta$, we consider
the set $\sigma_k^{\delta}=\{ z \in \C \ : \ |x_k(z)| \le 1+\delta
\}$. Notice that $x_k(z)$ is one-half the trace of the transfer
matrix of the potential under consideration from the origin to
$F_k$. This is why the case $\theta = 0$ is special! It follows
from \cite[Lemma~4]{dt3} that for $0 \le j \le k-1$, $\min \{
|x_j(z)| , |x_{j+1}(z)| \} \le 1 + \delta$ for every $z \in
\sigma_k^{\delta}$. Since, for every $j$, the potential repeats
its values from $1$ to $F_j$ once, we can apply
\cite[Lemma~2]{dt2} to bound the right-hand side of
\eqref{parseval0} from below. Of course we use either $j$ or
$j+1$, depending on which of $|x_j(z)$ and $|x_{j+1}(z)|$ is
bounded by $1 + \delta$. Carrying this out inductively, we find
that
$$
\sum_{n \ge N} |u(n)|^2 \ge C N^{2 \kappa} (|u(1)|^2 + |u(0)|^2)
$$
for a $\delta$-dependent constant
$$
\kappa < \frac{\log \frac{\sqrt{17}}{4}}{5 \log \phi}
$$
and a corresponding suitable constant $C > 0$. One checks that
$\kappa$ can be made arbitrarily close to $\frac{\log
\frac{\sqrt{17}}{4}}{5 \log \phi}$ by making $\delta$ sufficiently
small. The same reasoning can be applied to $\langle P_l(N,\cdot)
\rangle (T)$. In the end, one only needs to observe that
$$
|u(1)|^2+|u(0)|^2+|u(-1)|^2 \ge c (1 + (\Im F(E +
\tfrac{i}{T}))^2),
$$
uniformly in the energy, where $F$ is the Borel transform of the
spectral measure; compare \cite{dst,dt1}.

From this point on, one can proceed as in the proof of
Theorem~\ref{t.main3}. It then follows, with the same notation as
above, that
\begin{align}\label{xxx}
\langle P(c T^\frac{1}{s}, \cdot) \rangle (T) & \ge c
T^{-1+2\kappa/s} \int_{I_k} (1 + (\Im F(E + \tfrac{i}{T}))^2) \, dE \\
\nonumber & \ge c T^{-1+2\kappa/s} |I_k| \\
\nonumber & \ge c T^{-2+2\kappa/s}
\end{align}
and hence
$$
\langle \beta_{\delta_0}^-(p) \rangle \ge \frac{1}{s} -
\frac{2}{p}\left( 1 - \frac{\kappa}{s} \right)
$$
from which the result follows.
\end{proof}

\noindent\textit{Remark.} It is possible to improve the estimate
\eqref{betalower0} if one revisits the inequality \eqref{xxx} and
observes that the spectral measure is singular and the integral
over $I_k$ in question is larger than just $|I_k|$. This can be
made quantitative using ideas from \cite{dt2}.

\end{document}